\documentclass[final,times,1p,reqno,11pt]{elsarticle}
\usepackage{enumerate}
\usepackage{amsfonts}
\usepackage{amsmath}
\usepackage{amsthm}
\usepackage{amssymb}
\usepackage{latexsym}
\usepackage{multicol}
\usepackage{verbatim}
\usepackage{amscd}
\usepackage{hyperref}
\usepackage{pb-diagram}

\newtheorem*{cor}{Corollary}
\newtheorem*{lem}{Lemma}
\newtheorem*{prop}{Proposition}
\allowdisplaybreaks[2]

\theoremstyle{definition}

\theoremstyle{definition}
\newtheorem{thm}{Theorem}
\newtheorem*{conj}{Conjecture}

\newenvironment{pf}{\proof}{\endproof}
\makeatletter
\def\mydggeometry{\makeatletter\dg@YGRID=1\dg@XGRID=20\unitlength=0.003pt\makeatother}
\makeatother \theoremstyle{remark}

\numberwithin{equation}{section}

\DeclareMathOperator{\Ht}{ht}
\DeclareMathOperator{\ad}{ad}
\let\bwdg\bigwedge
\def\bigwedge{{\textstyle\bwdg}}

\makeatletter
\def\ps@pprintTitle{%
     \let\@oddhead\@empty
     \let\@evenhead\@empty
     \def\@oddfoot{\footnotesize\itshape%
       \hfill\today}%
     \let\@evenfoot\@oddfoot}
\makeatother
\journal{\relax}

\newcommand{\wt}{\operatorname{wt}}
\newcommand{\tensor}{\otimes}

\newcommand{\nc}{\newcommand}

\renewcommand{\Bbb}{\mathbb}
\nc\bomega{{\mbox{\boldmath $\omega$}}} \nc\bpsi{{\mbox{\boldmath
$\Psi$}}}
 \nc\balpha{{\mbox{\boldmath $\alpha$}}}
 \nc\bpi{{\mbox{\boldmath $\pi$}}}

\newcommand{\lie}[1]{\mathfrak{#1}}


 \nc{\Hom}{\operatorname{Hom}}
\nc{\End}{\operatorname{End}} \nc{\wh}[1]{\widehat{#1}}
\nc{\Ext}{\operatorname{Ext}} \nc{\ch}{\operatorname{ch}}
\nc{\ev}{\operatorname{ev}} \nc{\Ob}{\operatorname{Ob}}
\nc{\soc}{\operatorname{soc}} \nc{\rad}{\operatorname{rad}}
\nc{\head}{\operatorname{head}}

 \nc{\Cal}{\mathcal} \nc{\Xp}[1]{X^+(#1)} \nc{\Xm}[1]{X^-(#1)}
\nc{\on}{\operatorname} \nc{\Z}{{\mathbf Z}} \nc{\J}{{\mathcal J}}
\nc{\C}{{\mathbf C}} \nc{\Q}{{\mathbf Q}}

\nc{\N}{{\Bbb N}} \nc\boa{\mathbf a} \nc\bob{\mathbf b} \nc\boc{\mathbf c}
\nc\bod{\mathbf d} \nc\boe{\mathbf e} \nc\bof{\mathbf f} \nc\bog{\mathbf g}
\nc\boh{\mathbf h} \nc\boi{\mathbf i} \nc\boj{\mathbf j} \nc\bok{\mathbf k}
\nc\bol{\mathbf l} \nc\bom{\mathbf m} \nc\bon{\mathbf n} \nc\boo{\mathbf o}
\nc\bop{\mathbf p} \nc\boq{\mathbf q} \nc\bor{\mathbf r} \nc\bos{\mathbf s}
\nc\bou{\mathbf u} \nc\bov{\mathbf v} \nc\bow{\mathbf w} \nc\boz{\mathbf z}
\nc\boy{\mathbf y} \nc\ba{\mathbf A} \nc\bb{\mathbf B} \nc\bc{\mathbf C}
\nc\bd{\mathbf D} \nc\be{\mathbf E} \nc\bg{\mathbf G} \nc\bh{\mathbf H}
\nc\bi{\mathbf I} \nc\bj{\mathbf J} \nc\bk{\mathbf K} \nc\bl{\mathbf L}
\nc\bm{\mathbf M} \nc\bn{\mathbf N} \nc\bo{\mathbf O} \nc\bp{\mathbf P}
\nc\bq{\mathbf Q} \nc\br{\mathbf R} \nc\bs{\mathbf S} \nc\bt{\mathbf T}
\nc\bu{\mathbf U} \nc\bv{\mathbf V} \nc\bw{\mathbf W} \nc\bz{\mathbf Z}
\nc\bx{\mathbf x}
\begin{document}
\begin{frontmatter}
\title{Minimal affinizations as projective objects\tnoteref{fn1}}
\author[ucr]{Vyjayanthi Chari}
\address[ucr]{Department of Mathematics, University of California, Riverside, CA 92521}
\emailauthor{vyjayanthi.chari@ucr.edu}{V.~C.}
\author[ucr]{Jacob Greenstein}\emailauthor{jacob.greenstein@ucr.edu}{J.~G.}
\tnotetext[fn1]{Partially supported by the NSF grants DMS-0901253 (V.~C.) and DMS-0654421 (J.~G.)}

\begin{abstract}
We prove that the specialization to $q=1$ of a Kirillov-Reshetikhin module for 
an untwisted quantum affine algebra of classical type is 
projective  in a suitable  category. This  yields a uniform character formula for the 
Kirillov-Reshetikhin modules. 
We conjecture that these results  holds for specializations 
of minimal affinization with some restriction on the corresponding highest weight. We discuss the 
connection with the conjecture of  Nakai and Nakanishi on $q$-characters of minimal affinizations. We establish this conjecture 
in some special cases. This also leads us to conjecture  an alternating sum formula for Jacobi-Trudi determinants.
\end{abstract}

\end{frontmatter}

\section*{Introduction}

The study of finite-dimensional representations of quantum affine algebras has attracted a lot of attention over the last twenty years. Nevertheless, there are many natural questions which remain unanswered, for instance, the basic problem of determining the character of an irreducible representation of a quantum affine algebra. Part of the reason for the difficulty is that there are far too many irreducible representations and although  character is known for a generic representation, the character in the non--generic case is difficult to understand. A particular family of these non--generic modules are the Kirillov-Reshetikhin modules and more generally the family of minimal affinizations of a dominant integral weight.

In~\cite{KR} Kirillov and Reshetikhin conjectured the existence of certain simple modules over a quantum affine algebra whose tensor products (and in particular the modules
themselves) have prescribed decompositions as direct sums of simple modules over the quantized enveloping algebra of the underlying simple Lie algebra $\lie g$. They are parametrized by
pairs $(m,i)$ where $m$ is a positive integer and $i$ is a node of the Dynkin diagram of~$\lie g$. More generally, after~\cite{Ch1} one can consider minimal affinizations
of simple finite dimensional modules over $U_q(\lie g)$, namely, minimal (in some natural partial order) simple modules over the quantum affine algebra with the ``top'' part
isomorphic to the given simple finite dimensional $U_q(\lie g)$-module.
Kirillov-Reshetikhin modules and minimal affinizations were and still are being actively studied
(to name but a few, cf. \cite{Ch2,CMkir1,FL,FSS,Her1,M,NN1,Nak1}, see also \cite{CH} for extensive bibliography on the subject). One of their pleasant properties is that they admit specializations at $q=1$
which can be naturally regarded
as modules over the current algebra $\lie g\tensor \bc[t]$. Moreover when  $\lie g$ is of classical type they are actually modules  for  the truncated current algebra
$\lie g\tensor \bc[t]/(t^2)$ which is isomorphic to the semidirect product of $\lie g$ with its adjoint representation.

The specializations of Kirillov-Reshetikhin modules and minimal affinizations are indecomposable but no longer simple and it is only natural to ask whether they
have any distinguished
homological properties.
In the present work, we apply the methods developed in~\cite{CG,CG1} to construct, for each dominant weight $\lambda$ of~$\lie g$, a
natural Serre subcategory of the category of graded modules over the truncated current algebra. We describe the  projective cover of an arbitrary simple object in this category
explicitly in terms of generators and relations.  We are also able to
 obtain an alternating character formula, which expresses the (graded) character of
 the projective cover of a simple object in terms of  the character of the simple object and the (graded) characters of (finitely many) projectives following it in a natural partial order.
The formula involves coefficients which can be computed from analyzing the module structure of the exterior algebra of an abelian ideal (naturally associated with the irreducible object)
in a fixed Borel subalgebra of~$\lie g$.

We show in Proposition \ref{krp} that
any specialized  Kirillov-Reshetikhin module is a projective object in one of these subcategories. In general, we prove by using a result of \cite{M} and Theorem \ref{thm1} of this paper, that the specialization of a minimal affinization is a quotient of the projective cover of an irreducible object in the corresponding subcategory. A reformulation of the conjecture in \cite{M} is that the minimal affinization is isomorphic to the projective cover.

The characters of minimal affinizations are not known in general and there are conjectures of  Nakai and Nakanishi (\cite{NN1}) that they are given by Jacobi-Trudi type formulae. This was  established by~\cite{Her2}
for the type~$B$. Combining Theorem \ref{thm2} of this paper with the conjectures of \cite{M} and \cite{NN1}, we conjecture an alternating formula for the Jacobi-Trudi determinants, which can be stated in a purely combinatorial way. We are able to verify this conjecture in certain cases which
proves that specializations of minimal affinization are projective objects in our subcategories and allows us to
conjecture that this is always the case for $\lie g$ of classical type.

The paper is organized as follows. Section~\ref{MR} contains the main results and the necessary preliminaries. The main results are proven in Section~\ref{PT1}. 
Finally, Section~\ref{maff} contains a brief exposition of the necessary results on  Kirillov-Reshetikhin modules and minimal  affinizations and connects the existing literature to
the main results and conjectures  of the paper.

\subsection*{Acknowledgments}
We are grateful to D.~Hernandez for stimulating discussions and to T.~Nakanishi for bringing to our attention the results of~\cite{KT}. This work was finished while the
second author was visiting Max-Planck-Institut f\"ur Mathematik and he thanks the Institute for its hospitality.

\section{The main results}\label{MR}

\subsection{}\label{MR.10}
Let $\lie g$ be a finite-dimensional
 complex simple Lie algebra with  a fixed  Cartan subalgebra $\lie h$.   Let $R$  be the
  corresponding root system and  fix a set $\{\alpha_i: i\in I\}\subset\lie h^* $ (where $I=\{1,\dots,\dim\lie h\}$) of simple roots for $R$.
   The root lattice $Q$ is the $\bz$-span of the simple roots while $Q^+$  is the $\bz_+$-span of the simple roots, and $R^+= R\cap  Q^+$ denotes the set of positive roots in $R$.
   Given $i\in I$, let $\epsilon_i:Q^+\to \bz_+$ be the homomorphism of free semi-groups defined by setting $\epsilon_i(\alpha_j)=\delta_{i,j}$, $j\in I$,
   and define  $\Ht:Q^+\to\bz_+$ by $\Ht(\nu)=\sum_{i\in I}\epsilon_i(\nu)$, $\nu\in Q^+$.

   The restriction of  the Killing  form $\kappa:\lie g\times\lie g\to\Bbb C$   to $\lie h\times\lie h$ induces a non--degenerate bilinear form $(\cdot ,\cdot )$ on $\lie h^*$ and we let $\{\omega_i\,:\, i\in I\}\subset\lie h^*$ be the fundamental weights defined
by $2(\omega_j,\alpha_i)=\delta_{i,j}(\alpha_i,\alpha_i)$.
 Let $P$ (respectively  $P^+$) be the $\bz$- (respectively  $\bz_+$-) span of the $\{\omega_i:i\in I\}$ and note that $Q\subseteq P$. Given $\lambda,\mu\in P$ we say that $\mu\le \lambda$ if and only if $\lambda-\mu\in Q^+$.
 Clearly $\le $ is a partial order on $P$. The set $R^+$ has a unique maximal element with respect to this order which is
 denoted by $\theta$ and is called the highest root of $R^+$.

 \subsection{} \label{MR.20}
 Given $\alpha\in  R$, let $\lie g_\alpha\subset\lie g$ be the corresponding root space and define subalgebras
 $\lie n^\pm$ of $\lie g$ by $$\lie n^\pm=\bigoplus_{\alpha\in  R^+}\lie g_{\pm\alpha}.$$
 We have isomorphisms of vector spaces
  \begin{equation}\label{C:gtriang}\lie g=\lie n^-\oplus\lie h\oplus\lie n^+,\ \ \bu(\lie
g)\cong\bu(\lie n^-)\otimes\bu(\lie h)\otimes \bu(\lie
n^+),\end{equation}
where $\bu(\lie a)$ is the universal enveloping algebra of the Lie algebra~$\lie a$.
   For $\alpha\in  R^+$, fix elements $x^\pm_\alpha\in\lie g_{\pm\alpha}$ and $h_\alpha\in\lie h$ such that they span a Lie subalgebra of $\lie g$ isomorphic to $\lie{sl}_2$, i.e., we have $$[h_\alpha,x^\pm_\alpha]=\pm 2
   x^\pm_\alpha,\qquad [x^+_\alpha, x^-_\alpha]= h_\alpha, $$ and more generally, assume that the set $\{x^\pm_\alpha: \alpha\in R^+\}\cup\{h_{\alpha_i}: i\in I\}$ is a Chevalley basis for~$\lie g$.
We abbreviate $x_i^\pm:=x_{\alpha_i}^\pm$ and $h_i:=h_{\alpha_i}$, $i\in I$.

\subsection{} \label{MR.30}
Let $V$ be a $\lie h$-module. Given $\mu\in\lie h^*$, the set  $V_\mu=\{v\in V: hv=\mu( h)v,\, h\in\lie h\}$ is called the $\mu$-weight space of $V$ .  We say that $V$ is a weight
module and that $\wt V$ is the set of weights of $V$, if $$V=\bigoplus_{\mu\in\lie h^*}
V_\mu,\qquad \wt V=\{\mu\in\lie h^*: V_\mu\ne 0\}.$$
Let $\bz[\lie h^*]$ be the integral group algebra of $\lie h^*$ and for $\mu\in\lie h^*$, let $e(\mu)$ be the corresponding generator of $\bz[\lie h^*]$.
If all weight spaces of~$V$ are finite dimensional, define $\ch V\in\bz[\lie h^*]$  by  $$\ch V=\sum_{\mu\in\lie h^*}\dim V_\mu\, e(\mu).$$
Observe that for two such modules $V_1$ and $V_2$, we have $$\ch(V_1\oplus V_2)=\ch  V_1+\ch V_2,\qquad \ch  (V_1\otimes V_2)=\ch V_1 \ch V_2. $$

  \subsection{}\label{MR.40}  For $\lambda\in P^+$,
let $V(\lambda)$ be the $\lie g$-module  generated by an element
$v_\lambda$  with defining
relations:$$hv_\lambda=\lambda(h)v_\lambda,\quad h\in\lie h\qquad
x_{i}^+v_\lambda=0=(x_i^-)^{\lambda(h_i)+1} v_\lambda=0,\qquad i\in I.$$
It is well-known that
$$
\dim
V(\lambda)<\infty,\qquad \dim V(\lambda)_\lambda=1,\qquad \wt
V(\lambda)\subset\lambda-Q^+.
$$
An irreducible finite-dimensional
$\lie g$-module is isomorphic to $V(\lambda)$ for a unique
$\lambda\in P^+$ and any finite-dimensional $\lie g$-module is
completely reducible. In other words, if $\mathcal F(\lie g)$ is the category of finite-dimensional $\lie g$-modules with morphisms being maps of $\lie g$-modules, then $\mathcal F(\lie g)$ is a semisimple category and the isomorphism classes of simple objects are indexed by elements of~$P^+$.

\subsection{}\label{MR.60}
For the rest of the paper we will be concerned with the Lie algebra~$\lie g\ltimes \lie g_{\ad}$.
As a vector space $$\lie g\ltimes\lie
g_{\ad}=\lie g\oplus\lie g,$$ and the Lie bracket is given by
$$[(x,y),(x',y')]=([x,x'],[x,y']+[y,x']).$$ In particular if we
identify $\lie g$ (respectively, $\lie g_{\ad}$) with the subspace
$\{(x,0):x\in\lie g\}$ (respectively, $\{(0,y): y\in\lie g\}$), then $\lie
g_{\ad}$ is an abelian Lie ideal in $\lie g\ltimes\lie g_{\ad}$. Given~$x\in\lie g$,
set ~$(x)_{\ad}=(0,x)\in\lie g_{\ad}$.

Define a $\bz_+$-grading on $\lie g\ltimes \lie g_{\ad}$ by
requiring the elements of $\lie g$ to have degree zero and elements
of $\lie g_{\ad}$ to have degree one. Then the universal enveloping
algebra  $\bu(\lie g\ltimes\lie g_{\ad})$ is a $\bz_+$-graded
algebra and as a trivial consequence of the PBW theorem, there is
an isomorphism of vector spaces $$\bu(\lie g\ltimes\lie
g_{\ad})\cong \bs(\lie g)\tensor \bu(\lie g),$$
where $\bs(\lie g)$ is the symmetric algebra of~$\lie g$.

\subsection{}\label{MR.70} Let $\mathcal G_2$ be the category whose objects are
 $\bz_+$-graded $\lie g\ltimes\lie g_{\ad}$-modules with finite-dimensional graded pieces and where the
 morphisms are $\lie g\ltimes\lie g_{\ad}$-module maps which preserve the grading.
 In other words a $\lie g\ltimes\lie g_{\ad}$-module $V$ is an object of $\mathcal G_2$
 if and only if
\begin{gather*}
V=\bigoplus_{k\in\bz_+} V[k],\qquad\dim V[k]<\infty,\\
 \lie g V[k]\subset V[k],\qquad \lie g_{\ad }V[k]\subset
 V[k+1],
\end{gather*}
and if $V,W\in\Ob\mathcal G_2$, then
$$
\Hom_{\mathcal G_2}(V,W)=\{f\in\Hom_{\lie g\ltimes\lie g_{\ad}}(V,W): f(V[k])\subset W[k]\}.
$$
For $r\in\bz_+$, let $\tau_r$ be the grading shift functor on $\mathcal G_2$: thus $\tau_rV$ has the same $\lie g\ltimes\lie g_{\ad}$-module structure as $V$ but the grading is uniformly shifted by $r$.
Given a $\lie g$-module $V$, define $\ev_r V\in\mathcal G_2$  by $$(\ev_r V)[s]=\begin{cases} 0,& s\ne r,\\ V,& s=r,\end{cases} \qquad
(x,y)(v)=xv,\quad x,y\in\lie g, v\in V
$$
and observe that $\ev_r V=\tau_r\ev_0V$. The following is easily checked.
\begin{lem}\label{simpleg2} A simple object in $\mathcal G_2$ is isomorphic to $\ev_r V(\lambda)$ for some $r\in\bz_+$ and $\lambda\in P^+$. Moreover if $\nu,\mu\in P^+$ and $p,s\in\bz_+$,  we have $$\ev_p V(\nu)\cong \ev_s V(\mu)\iff p=s, \nu=\mu.
$$
Equivalently, the isomorphism classes of simple objects in $\mathcal G_2$ are indexed by the set
$\Lambda=P^+\times\bz_+.$ \qed \end{lem}
Given $(\lambda,r)\in \Lambda$, set
$$[V:\ev_r V(\lambda)]=\dim \Hom_{\lie g}(V(\lambda), V[r]).$$ If $V$
is finite-dimensional, then $[V: \ev_r V(\lambda)]$ is just the
multiplicity of $\ev_r V(\lambda)$ in  a Jordan-Holder series for $V$.
Define the graded character $\ch_t$ of $V\in\Ob{\mathcal G_2}$ by $$\ch_t V=\sum_{r\in\bz_+}t^r\ch V[r]=\sum_{(\lambda,r)\in\Lambda} t^r [V:\ev_r V(\lambda)]\ch V(\lambda)
$$
and observe that this is a well defined element of $\bz[\lie h^*][[t]]$. Clearly, $\ch_t \tau_r V=t^r \ch_t V$.

 \subsection{}\label{ggammadef}\label{MR.80} Given any subset $\Gamma$ of $\Lambda$, let $\mathcal G_2[\Gamma]$ be the full subcategory of $\mathcal G_2$ defined by, \begin{equation}\label{ggamma}V\in\Ob\mathcal G_2[\Gamma]\iff [V:\ev_r V(\lambda)]\ne 0\implies (\lambda,r)\in\Gamma\end{equation}
Lemma~\ref{MR.70} implies that $\mathcal G_2=\mathcal G_2[\Lambda]$ and also that  the isomorphism classes of simple objects in $\mathcal G_2[\Gamma]$ are indexed by elements of $\Gamma$. Given $V\in\Ob \mathcal G_2$, denote by $V^\Gamma$ the maximal quotient of $V$ which lies in $\mathcal G_2[\Gamma]$. It is standard that $V^\Gamma$ is well--defined and it is possible that $V^\Gamma$ is zero. It is clear from the definition that if $\Gamma'\subset\Gamma$, then $V^{\Gamma'}$ is a quotient of $V^\Gamma$ and also that if $V$ is a projective object of  $\mathcal G_2[\Gamma]$ then $V^{\Gamma'}$ is a projective object of $\mathcal G_2[\Gamma']$.

Given $\Gamma\subset \Lambda$ and $r\in\bz_+$, we set $\tau_r \Gamma=\{ (\mu,r+s)\,:\, (\mu,s)\in\Gamma\}$. Clearly, if $V\in\Ob\mathcal G_2[\Gamma]$ then
$\tau_r V\in\Ob\mathcal G_2[\tau_r\Gamma]$.

 \subsection{}\label{projLambda} The category $\mathcal G_2[\Gamma]$ is in general not semisimple and one of the goals of this paper is to understand the structure of
 the projective covers of simple objects in $\mathcal G_2[\Gamma]$ for suitable subsets $\Gamma$. We recall some results from \cite{CG} and \cite{CG1}. For $(\lambda,r)\in P^+$, set
 \begin{equation}\label{projLambda.20}
 P(\lambda,r)=\bu(\lie
g\ltimes \lie g_{\ad})\otimes_{\bu(\lie g)}\ev_r V(\lambda).
\end{equation} The following was proved in \cite[Proposition~2.2]{CG1}.
\begin{prop}Let $(\lambda,r)\in\Lambda$.
\begin{enumerate}[{\rm(i)}]
\item\label{projLambda.i}  The object $P(\lambda,r)$ is the projective cover   in $\mathcal G_2$  of its unique irreducible quotient $\ev_rV(\lambda)$.
Moreover, the kernel of the canonical morphism $P(\lambda,r)\to \ev_r V(\lambda)$ is generated by $P(\lambda,r)[r+1]$.
\item\label{projLambda.ii} For $(\mu,s)\in\Lambda$, we have
$$
[P(\lambda,r):\ev_sV(\mu)]=\dim\Hom_{\lie g}(\bs^{s-r}(\lie
g)\otimes V(\lambda), V(\mu)).
$$
\item\label{projLambda.iii}  $P(\lambda,r)$ is
the $\lie g\ltimes\lie g_{\ad}$-module generated by the element
$p_{\lambda,r}=1\otimes v_\lambda$ with defining relations:$$\lie n^+ p_{\lambda,r} =0,\ \
hv_\lambda=\lambda(h)p_{\lambda,r}, \qquad
(x^-_{i})^{\lambda(h_i)+1}p_{\lambda,r}=0,$$ for all $h\in\lie h$
and $i\in I$. \qed
\end{enumerate}
\end{prop}
One way of summarizing the first main result of this paper is the following: to give an analog of the preceding proposition for the projective covers of simple objects of $\mathcal G_2[\Gamma]$. As we have remarked in  Section~\ref{ggammadef}, the module $P(\lambda,r)^\Gamma$ is projective in $\mathcal G_2[\Gamma]$.  If $(\lambda,r)\in\Gamma$ then $P(\lambda,r)^\Gamma$ is  projective with unique irreducible quotient $\ev_r V(\lambda)$. In particular, $P(\lambda,r)^\Gamma$ is indecomposable and is the projective cover of $\ev_r V(\lambda)$. In other words, one has the analog of  Proposition~\ref{projLambda}\eqref{projLambda.i}  for all subsets $\Gamma$ of $\Lambda$.

\subsection{} The analog of Proposition \ref{projLambda}\eqref{projLambda.ii} was also studied in~\cite{CG} and~\cite{CG1}. This time the result is not true for an arbitrary subset of $\Lambda$ and one way of identifying subsets for which this remained true was to introduce a strict partial order on $\Lambda$ and to consider interval closed subsets in this order. To explain this, given~$(\lambda,r), (\mu,s)\in\Lambda$,  we say that
$(\mu,s)$ covers $(\lambda,r)$ if and only if $s=r+1$
and~$\mu-\lambda\in R\,\sqcup\{0\}$. It follows immediately that for
any~$(\mu,s)\in\Lambda$ the set of~$(\lambda,r)\in\Lambda$ such
that~$(\mu,s)$ covers $(\lambda,r)$ is finite. Let $\preccurlyeq$ be
the unique partial order on $\Lambda$ generated by this cover
relation. A subset $\Gamma$ of
$\Lambda$ is called interval closed if for all
$\gamma\preccurlyeq\gamma'\in\Gamma$, we have
$[\gamma,\gamma']:=\{\gamma''\in \Lambda\,:\,\gamma\preccurlyeq\gamma''\preccurlyeq\gamma'\}\subset \Gamma$. The following result is a straightforward reformulation of~\cite[Propositions~2.6 and~2.7]{CG}.
\begin{prop}\label{analogii} Let $\Gamma$ be an interval closed subset of $\Lambda$ and let $(\lambda,r)\in
\Gamma$. Then
\begin{equation*}
[P(\lambda,r)^
\Gamma:\ev_s V(\mu)]=\begin{cases}\dim\Hom_{\lie g}(\bs^{s-r}(\lie
g)\otimes V(\lambda), V(\mu)),& (\mu,s)\in\Gamma,\\ 0,&(\mu,s)\notin\Gamma.\mskip110mu\mbox{\qedsymbol}\end{cases}
\end{equation*}
\end{prop}
 \subsection{}\label{MR.100} We now state one of  the main results of this paper. The first result   establishes the analog of  Proposition~\ref{projLambda}
 by imposing a further restriction on $\Gamma$. Given $\lambda\in P^+$ and any subset $S$ of $R$, define  $\Gamma(\lambda,S)\subset\Lambda$ by
 $$\Gamma(\lambda,S)=\{ (\mu,r)\in \Lambda\,:\,
\mu=\lambda-\sum_{\beta\in S} n_\beta \beta,\,n_\beta\in\bz_+,\,\sum_{\beta\in S}n_\beta=r\},$$ and note that the subset $\Gamma(\lambda,S)$ need not be interval closed in general.
 For $(\mu,r)\in\Gamma$, let $p_{\mu,r}^\Gamma$ be the image of $p_{\mu,r}$ in $P(\mu,r)^\Gamma$.
\begin{thm} \label{thm1} Let $\Psi\subset R^+$ be such that
$$
\Psi=\{ \alpha\in R\,:\, (\alpha,\xi)=\max_{\beta\in R} (\beta,\xi)\},\qquad\text{for some~$\xi\in P^+$},
$$
and let $\Gamma=\Gamma(\lambda,\Psi)$ for some~$\lambda\in P^+$. Then $\Gamma$ is a finite interval closed subset of $\Lambda$ and for all
 $(\mu,r)\in\Gamma$ the module   $P(\mu,r)^\Gamma$ is   finite-dimensional and is the projective cover of $\ev_r V(\mu)$ in $\mathcal G_2[\Gamma]$. Moreover, it is  generated by the element $p_{\mu,r}^\Gamma$ with defining relations,
\begin{gather}\label{rel1}\lie n^+\  p^\Gamma_{\mu,r} =0,\qquad
hp^\Gamma_{\mu,r}=\mu(h)p^\Gamma_{\mu,r},\quad h\in\lie h,\qquad
(x^-_{i})^{\mu(h_i)+1} p^\Gamma_{\mu,r}=0,\quad  i\in I,\\
\label{rel2} \lie n^+_{\ad} p_{\mu,r}^\Gamma=0=\lie h_{\ad} p_{\mu,r}^\Gamma,\quad
(x^-_\alpha)_{\ad}\  p_{\mu,r}^\Gamma=0,\quad\alpha\in R^+\setminus\Psi.\end{gather}  \end{thm}
 \subsection{}\label{PT1.110} The  following Corollary to Theorem \ref{thm1} is immediate.
  \begin{cor}\label{corthm1} Let $(\mu,r)\in\Gamma(\lambda,\Psi)$. Then $$\tau_rP(\mu,0)^{\Gamma(\mu,\Psi)}\cong P(\mu,r)^\Gamma.$$
  \end{cor}

\subsection{}\label{MR.120} Our next main result establishes an alternating character formula for the modules $P(\mu,r)^\Gamma$ under the same restrictions on $\Gamma$ as in Theorem \ref{thm1}.
Given $S\subset R^+$, set
$$
\lie n^\pm_S=\bigoplus_{\alpha\in S}\lie g_{\pm\alpha}.
$$
Clearly $\lie n^\pm_S$ is a weight $\lie h$-module and moreover if $k\in\bz_+$ then $\bigwedge^k\lie n^\pm_S$ is also a weight $\lie h$-module.
If  $\Psi\subset R^+$  is an in Theorem~\ref{thm1},
then it is easy to see that $\lie n^\pm_\Psi$ is an (abelian) Lie ideal in $(\lie h\oplus\lie n^\pm)$ and hence $\bigwedge^k \lie n^\pm_\Psi$ is
a $(\lie h\oplus\lie n^\pm)$-module. Moreover, ~$\lie n^\pm$ acts nilpotently on~$\bigwedge^k\lie n^\pm_\Psi$.

\begin{thm}\label{thm2} Let $\Psi=\{ \alpha\in R\,:\, (\alpha,\xi)=\max_{\beta\in R} (\beta,\xi)\}$ for some $\xi\in P^+$.
For all $\lambda\in P^+$, we have
$$
\sum_{(\nu,s)\in\Gamma(\lambda,\Psi)}(-t)^{s} c^\lambda_{\nu,s}\ch_t P(\nu,0)^{\Gamma(\nu,\Psi)}=\ch V(\lambda),
$$
where
$$
 c^\lambda_{\nu,s}=\dim\{ v\in (\bigwedge^s \lie n^-_\Psi)_{\nu-\lambda}\,:\, (x_i^-)^{\nu(h_i)+1}(v)=0,\, \forall\, i\in I\}.
  $$

\end{thm}

\subsection{}\label{subs:NN conj} Suppose that $\lie g$ is of type $B_n$, $C_n$ or~$D_{n+1}$. Assume that the nodes of the Dynkin diagram are labeled as in \cite{Bo}.  Assume also
 that $\lambda\in P^+$ is such that $\lambda(h_i)=0$, $i\ge n$. Set
 \begin{align*}
&i_\lambda=\max\{i \in I\,:\,\lambda(h_i)>0\}
\\
\intertext{and let}
&\Psi_\lambda=\{\alpha\in R^+\,:\, \epsilon_{i_\lambda}(\alpha)=2\}.
\end{align*}
Then either $\Psi_\lambda=\emptyset$ or $\Psi_\lambda=\{\alpha\in R^+\,:\, (\alpha,\omega_{i_\lambda})=\max_{\beta\in R} (\beta,\omega_{i_\lambda})\}$.
For $1\le i\le i_\lambda$, set $$
\lambda_i=\sum_{i\le k\le i_\lambda}\lambda(h_k),$$ and define the
 the   Jacobi-Trudi determinant $\bh_\lambda$ corresponding to $\lambda$ by
\begin{align*}
  &\bh_\lambda =\det(\boh_{\lambda_i-i+j})_{1\le i,j\le i_\lambda}\in\bz[P],\\
&\boh_k=
\begin{cases}
     \ch V(k\omega_1),& \lie g=B_n,D_{n+1}\\
     \sum_{0\le r\le k/2} \ch V((k-2r)\omega_1),&\lie g=C_n,
    \end{cases}
\end{align*}
where we adopt the convention that $\boh_k=0$ if $k<0$.
The following conjecture is motivated by Theorems \ref{thm1} and \ref{thm2} of this paper along with the conjectures in \cite{M} and \cite{NN1}. We will explain this further in the last section of the paper.
Recall that if  $V\in\Ob\mathcal G_2$ is finite-dimensional then $\ch V$ is just the specialization of $\ch_t V$ at $t=1$.
\begin{conj}\label{cgconj} Let $\lie g$ be of type $B_n$, $C_n$ or~$D_{n+1}$ and let $\lambda\in P^+$ be such that $\lambda(h_i)=0$, $i\ge n$.
 Then $$\ch P(\lambda,0)^{\Gamma(\lambda,\Psi_\lambda)}=\bh_{\lambda},$$ or equivalently by Theorem \ref{thm2}, \begin{equation}\label{nnmcg}
\sum_{(\nu,s)\in\Gamma(\lambda,\Psi_\lambda)}(-1)^{s} c^\lambda_{\nu,s}\bh_\nu=\ch V(\lambda).
\end{equation}
\end{conj}
\medspace
 Our final result is the following.
\begin{prop}\label{thm3} The conjecture is true if one of the following holds
\begin{enumerate}[{\rm(i)}]
\item\label{thm3.i} $\lambda=m\omega_i$ for $i\in I$,
\item\label{thm3.ii} $i_\lambda\le 5$.
\end{enumerate}
\end{prop}

\section{Proof of Theorems~\ref{thm1} and~\ref{thm2}} \label{PT1}
Throughout this section we fix $\xi,\lambda\in P^+$ and set
$$
\Psi=\{ \alpha\in R\,:\, (\alpha,\xi)=\max_{\beta\in R} (\beta,\xi)\}\subset R^+
,\quad \Gamma=\Gamma(\lambda,\Psi).$$
The following property of $\Psi$ (cf.~\cite[Lemma~2.3]{CG1}) is crucial for this section.
Suppose that
$$
\sum_{\alpha\in R} m_\alpha \alpha=\sum_{\beta\in\Psi} n_\beta \beta,\qquad m_\alpha,n_\beta\in\bz_+.
$$
Then
\begin{equation}\label{rigid}
\sum_{\beta\in\Psi} n_\beta\le \sum_{\alpha\in R} m_\alpha,\ \ {\rm and}\ \ \sum_{\beta\in\Psi} n_\beta =\sum_{\alpha\in R} m_\alpha \implies\  m_\alpha=0,\ \  \alpha\notin\Psi.
 \end{equation}

\subsection{} \label{PT1.10}
We begin by proving that $\Gamma$ is interval closed.
As the first step, we prove the following Lemma.  It is tempting to think that this result  is clear from the definition of the partial order. As an example, suppose that $\mu=\lambda-\alpha-\beta$ where $\alpha,\beta\in R$ and $\alpha+\beta\notin R\cup\{0\}$,
then one is inclined to believe that $(\lambda,0)\prec(\mu,2)$. However there is a subtlety here, namely the partial order requires the existence of $\nu\in P^+$
such that $(\lambda,0)\prec(\nu,1)\prec(\mu,2)$ which is not obvious, since it could happen very easily that $\lambda-\alpha,\lambda-\beta\notin P^+$.
The proof of this Lemma also shows that for an arbitrary $S$ the set $\Gamma(\lambda,S)$ cannot be expected to be interval closed.
\begin{lem}
Let $\lambda\in P^+$, $S\subset R$. Then
$$ (\mu,r)\in\Gamma(\lambda,S)\implies
(\lambda,0)\preccurlyeq (\mu,r).
$$
\end{lem}

\begin{pf} Since $\Gamma(\lambda,S)\subset \Gamma(\lambda, R)$,
it suffices to prove the Lemma for~$\Gamma(\lambda, R)$. We proceed
 by induction on~$r$ and note that induction begins since the result is obviously true for $r\le 1$.
 For the inductive step, we may assume
 that $(\lambda,0)\preccurlyeq (\nu,s)$ for all $(\nu,s)\in \Gamma(\lambda,R)$ with $s<r$.
Let $(\mu,r)\in\Gamma(\lambda,R)$. If $\lambda-\mu$ can be written as a sum of $s<r$ roots then $(\mu,s)\in\Gamma(\lambda,R)$ and hence we have
$(\lambda,0)\preccurlyeq (\mu,s)\prec (\mu,r)$.

Assume now that $\lambda-\mu$ cannot be written as a sum of $s$ roots from $R$ with $s<r$. Choose $\beta_k\in R$, $1\le k\le r $ with  $\Ht(\beta_1)\ge\Ht(\beta_k)$ such that
  $\lambda=\mu-\sum_{j=1}^r \beta_j$.  We claim that there exists another expression $\lambda-\mu=\sum_{j=1}^r \gamma_j$
with $\gamma_j\in R$, $\Ht(\gamma_1)\ge \Ht(\gamma_s)$, $1\le s\le r$ and such that  $\mu+\gamma_1\in P^+$. Note that this
claim implies the assertion since then $(\mu+\gamma_1,r-1)\in \Gamma(\lambda,R)$ and the induction hypothesis yields
$$
(\lambda,0)\prec (\mu+\gamma_1,r-1)\prec (\mu,r).
$$
To prove the claim, we use
downward  induction on $\Ht(\beta_1)$. The induction starts when $\Ht(\beta_1)$ is maximal since then $\beta_1=\theta\in P^+$
and so~$\mu+\beta_1\in P^+$.
For the inductive step, if~$\mu+\beta_1\notin P^+$,
choose~$i\in I$ such that $(\mu+\beta_1)(h_i)<0$ in which case we have~$\beta_1(h_i)<0$ and so $\beta_1+\alpha_i\in
R\cup\{0\}$.

On the other hand, since $\lambda\in P^+$, it follows that there exists $1< s\le r$ such that $$\Big(\mu+\sum_{k=1}^{s-1}\beta_k\Big)(h_{i})<0,\qquad \Big(\mu+\sum_{k=1}^{s}\beta_k\Big)(h_{i})
\ge 0.
$$
Then $\beta_{s}(h_{i})>0$ and hence $\beta_s-\alpha_{i}\in R\cup\{0\}$. Setting $$\gamma_1=\beta_1+\alpha_i,\qquad\gamma_s=\beta_s-\alpha_i,\qquad  \gamma_j=\beta_j,\quad  2\le j\le r,\,  j\ne s,$$
we see that
$\lambda-\mu=\sum_{j=1}^r \gamma_j$.
If $\gamma_1=0$ or $\gamma_s=0$ then $\lambda-\mu$ can
be written as a sum of less than~$r$ roots which is a contradiction. Thus, $\gamma_1,\gamma_s\in R$ and
we have obtained an expression of~$\lambda-\mu$ as a sum of~$r$ roots with $\Ht(\gamma_1)\ge \Ht(\gamma_s)$, $1\le s\le r$ and $\Ht(\gamma_1)>\Ht(\beta_1)$.
The  induction hypothesis applies to this new expression and hence  completes the proof of the claim.
\end{pf}

\subsection{}\label{PT1.30}
\begin{prop}
\begin{enumerate}[{\rm(i)}]
\item\label{PT1.30.i} Given $\mu\in P^+$, there exists at most one $s\in\bz_+$ such that $(\mu,s)\in\Gamma$.
\item\label{PT1.30.ii} The subset $\Gamma$ is interval closed and finite.
\item\label{PT1.30.iii} Suppose that $(\mu,r), (\nu,s)\in\Gamma$ and $(\mu,r)\prec (\nu,s)$. Then $(\nu,s-r)\in \Gamma(\mu,\Psi)$.

    \end{enumerate}
\end{prop}
\begin{pf}
Part~\eqref{PT1.30.i} is immediate from \eqref{rigid}. To prove~\eqref{PT1.30.ii}, by Lemma~\ref{PT1.10} it is enough to show that if  $(\mu,r)\in\Lambda$ and
$(\nu,s)\in\Gamma$ satisfy
$(\lambda,0)\prec(\mu,r)\prec(\nu,s)$ then
$(\mu,r)\in\Gamma$. We can write,
$$\lambda-\mu=\sum_{\alpha\in R} n_\alpha\alpha,\qquad\mu-\nu=\sum_{\alpha\in R} k_\alpha\alpha,\qquad\lambda-\nu=\sum_{\beta\in \Psi} m_\beta\beta,$$ with $n_\alpha, k_\alpha, m_\beta\in\bz_+$  such that  $$\sum_{\alpha\in R} n_\alpha\le r,\qquad\sum_{\alpha\in R}k_\alpha\le s-r,\qquad \sum_{\beta\in\Psi}m_\beta=s,$$
and hence we get $$
\sum_{\beta\in\Psi} m_\beta \beta=\sum_{\alpha\in R}( n_\alpha+k_\alpha)\alpha,\qquad
\sum_{\alpha\in R}(n_\alpha+k_\alpha)\le s=\sum_{\beta\in\Psi} m_\beta.$$
Using \eqref{rigid} we see that this implies that $$n_\alpha=k_\alpha=0,\quad \alpha\notin\Psi,\qquad \sum_{\alpha\in R}n_\alpha=r,\quad \sum_{\alpha\in R}k_\alpha=s-r,$$ which proves simultaneously that  $(\mu,r)\in\Gamma$ and $(\nu,s)\in\tau_r\Gamma(\mu,\Psi)$ or equivalently $(\nu,s-r)\in\Gamma(\mu,\Psi)$. Note that the set $(\lambda-Q^+)\cap P^+$ is finite and now using part~\eqref{PT1.30.i} we see that $\Gamma$ is finite.
\end{pf}

\subsection{}\label{PT1.40} It follows from Proposition~\ref{analogii} and  Proposition~\ref{PT1.30}\eqref{PT1.30.ii} that $P(\mu,r)^\Gamma$ is finite-di\-men\-sion\-al.
To complete the proof of Theorem~\ref{thm1} we must  determine the defining relations for $P(\mu,r)^\Gamma$. As a first step, we have,\begin{lem}
The generator $p_{\mu,r}^\Gamma$ of $P(\mu,r)^\Gamma$ satisfies the relations~\eqref{rel1} and~\eqref{rel2}.
\end{lem}
\begin{pf}
Since $p_{\mu,r}\in P(\mu,r)$  satisfies~\eqref{rel1} in $P(\mu,r)$ it follows that they also hold for its image $p_{\mu,r}^\Gamma$.
We prove that $p_{\mu,r}^\Gamma$ satisfies the relations~\eqref{rel2}. Suppose first that $(\lie n^+)_{\ad} p_{\mu,r}^\Gamma\not=0$ and choose
$\alpha\in R^+$ such that $\Ht(\alpha)$ is maximal with the property $(x^+_\alpha)_{\ad}p_{\mu,r}^\Gamma\ne 0$. Then $$\lie n^+\left((x^+_\alpha)_{\ad}p_{\mu,r}^\Gamma\right)=0.$$ It follows from the standard representation theory of simple Lie algebras that $\mu+\alpha\in P^+$ and since $P(\mu,r)^{\Gamma}\in\Ob\mathcal G_2[\Gamma]$, we conclude that
$$(\mu+\alpha,r+1)\in\Gamma,\qquad (\mu,r)\prec(\mu+\alpha,r+1).$$
Using Proposition~\ref{PT1.30}\eqref{PT1.30.iii}, we see that this forces $(\mu+\alpha,1)\in \Gamma(\mu,\Psi)$.  But  since $\alpha\in R^+$, this is impossible by the definition of $\Gamma(\mu,\Psi)$. Hence we have a contradiction and we have proved that \begin{equation}\label{nad}(\lie n^+)_{\ad}p_{\mu,r}^\Gamma=0.\end{equation}

Next, suppose that $(h)_{\ad}p_{\mu,r}^\Gamma\not=0$ for some $h\in\lie h$. Since $[\lie n^+,\lie h_{\ad}]\subset\lie n^+_{\ad}$, we see by using \eqref{nad} that  $$\lie n^+(h)_{\ad}p_{\mu,r}^\Gamma=0,$$
 hence $(\mu,r+1)\in\Gamma$. But since $(\mu,r)\in\Gamma$ this is impossible by Proposition~\ref{PT1.30}\eqref{PT1.30.i}, and so we have
 \begin{equation}\label{had}(\lie h)_{\ad}p_{\mu,r}^\Gamma=0.\end{equation}

 Finally, if $(\lie n^-_{R^+\setminus\Psi})_{\ad} p_{\mu,r}^\Gamma\not=0$,
 choose $\alpha\in R^+\setminus\Psi$ with $\Ht(\alpha)$ minimal such that $(x^-_\alpha)_{\ad}p_{\mu,r}^\Gamma\ne 0$. Suppose that there exists $\beta\in R^+$ such that
$$x_\beta^+(x^-_\alpha)_{\ad}p_{\mu,r}^\Gamma\ne 0.$$ By \eqref{nad} and \eqref{had} we see that this implies
  $\alpha-\beta\in R^+$ and in fact that
  $$(x^-_{\alpha-\beta})_{\ad}p_{\mu,r}^\Gamma= [x_\beta^+,(x^-_\alpha)_{\ad}]p_{\mu,r}^\Gamma\ne 0.$$
The minimality of $\Ht(\alpha)$ now forces $\alpha-\beta\in\Psi$, and since $$(\lambda,\alpha)=(\lambda,\alpha-\beta)+(\lambda,\beta)\ge (\lambda,\alpha-\beta)$$  the definition of $\Psi$ forces $\alpha\in\Psi$. But this contradicts our assumption that $\alpha\in R^+\setminus\Psi$. Thus, $\lie n^+ (x^-_\alpha)_{\ad}p_{\mu,r}^\Gamma=0$, which implies that $(\mu-\alpha,r+1)\in\Gamma$ and so
$(\mu-\alpha,1)\in\Gamma(\mu,\Psi)$ by Proposition~\ref{PT1.30}\eqref{PT1.30.ii}. Since~$\alpha\notin\Psi$, we obtain a contradiction.
\end{pf}

\subsection{}\label{PT1.50} To complete the proof of Theorem \ref{thm1} we must show that \eqref{rel1} and \eqref{rel2} are the defining relations.
\begin{prop}\label{defining}
Let $(\mu,r)\in \Gamma$ and suppose that  $\bp_\Psi(\mu,r)\in\Ob\mathcal G_2$ is the $\lie g\ltimes\lie g_{\ad}$-module generated by an element $\bop_{\mu,r}$ with grade $r$ and (graded)  defining relations: \begin{gather}\label{rel1a}\lie n^+  \bop_{\mu,r} =0,\qquad
h\bop_{\mu,r}=\mu(h)\bop_{\mu,r},\quad h\in\lie h,\qquad
(x^-_{\alpha_i})^{\mu(h_i)+1} \bop_{\mu,r}=0,\quad  i\in I,\\
\label{rel1b}
 \lie n^+_{\ad}\bop_{\mu,r}=0=\lie h_{\ad} \bop_{\mu,r},\quad
(x^-_\alpha)_{\ad}\ \bop_{\mu,r}=0,\quad\alpha\in R^+\setminus\Psi.\end{gather}
Then $\bp_\Psi(\mu,r)$ is an idecomposable
object in~$\mathcal G_2[\Gamma]$ and $\ev_r V(\mu)$ is its unique simple quotient.
\end{prop}

Assuming this proposition the proof of Theorem~\ref{thm1} is completed as follows.
\begin{pf}[Proof of Theorem~\ref{thm1}]

By Lemma~\ref{PT1.40}, $p_{\mu,r}^\Gamma$ satisfies the defining relations~\eqref{rel1a} and~\eqref{rel1b} hence
the assignment $\bop_{\mu,r}\mapsto p_{\mu,r}^\Gamma$ defines a surjective morphism $\phi:\bp_\Psi(\mu,r)\to P(\mu,r)^\Gamma$ in~$\mathcal G_2$.
Let $K$ be the kernel of this morphism so that we have a short exact sequence $$0\to K\to\bp_\Psi(\mu,r)\to P(\mu,r)^\Gamma\to 0.$$
Since $P(\mu,r)^\Gamma$ is projective  in $\mathcal G_2[\Gamma]$ while $\bp_\Psi(\mu,r)\in\Ob\mathcal G_2[\Gamma]$ by the above Proposition,
this sequence splits. Since  $\bp_\Psi(\mu,r)$ is indecomposable, it follows that $K=0$ and the theorem is proved.
\end{pf}

\subsection{}\label{PT1.70} We shall need the following elementary result on integrable representations of $\lie g$.

\begin{lem} Suppose that $V$ is an integrable $\lie g$-module, that is
for all $v\in V$ we have $\dim\bu(\lie g)v<\infty$. Then $V$ is isomorphic to a direct sum of simple $\lie g$-modules.
Moreover, $$0\not=v\in V_\mu,\qquad v\in\bu(\lie n^-)\Big(\bigoplus_{\nu\in\lie h^*\,:\, \mu<\nu} V_\nu\Big)\, \implies \, \lie n^+v\ne 0.$$
\end{lem}
\begin{pf} By assumption, $V$ is a sum of finite dimensional $\lie g$-modules. Since any finite-dimensional $\lie g$-module is semisimple,
the first statement follows say by~\cite[\S XVII.2]{L}.
To prove the second assertion, write
$$
V=\bigoplus_{\lambda\in P^+}V[\lambda],
$$
where $V[\lambda]$ is the isotypical component  of~$V$  corresponding to $\lambda\in P^+$.  Suppose that $$0\not=v\in V_\mu,\ \
\lie n^+ v=0. $$ Then since $\bu(\lie g)v$ is finite-dimensional, it follows that   $\bu(\lie g)v\cong V(\mu)$ and hence that  $v\in V[\mu]$.
Since $$\mu<\nu\implies V_\nu\subset \bigoplus_{\lambda\in P^+\,:\,\mu<\lambda} V[\lambda],$$ and $$V[\mu]\ \cap\ \bigoplus_{\lambda\in P^+\,:\,\mu<\lambda} V[\lambda]=\{0\},$$ it follows that $$v\notin\bu(\lie n^-)\Big(\bigoplus_{\nu\in\lie h^*\,:\, \mu<\nu} V_\nu\Big),$$ which proves the Lemma.
\qedhere
\end{pf}

\subsection{}\label{PT1.90} We now prove Proposition \ref{defining}.
\begin{pf}[Proof of Proposition~\ref{defining}]
It is clear that $\ev_r V(\mu)$ is a quotient of $\bp_\Psi(\mu,r)$ and hence $\bp_\Psi(\mu,r)$ is non--zero. Moreover, since the sum of
proper submodules of~$\bp_\Psi(\mu,r)$ is again a proper submodule, $\bp_\Psi(\mu,r)$ has the unique maximal proper submodule
and therefore
$\ev_r V(\mu)$ is
the unique simple quotient of $\bp_\Psi(\mu,r)$. In particular, $\bp_\Psi(\mu,r)$ is indecomposable.

It remains to prove that $\bp_\Psi(\mu,r)$ is an object in~$\mathcal G_2[\Gamma]$ provided that $(\mu,r)\in\Gamma$.
By Proposition~\ref{PT1.30}\eqref{PT1.30.iii} we have $\tau_r\Gamma(\mu,\Psi)\subset\Gamma$. Since $$\tau_r\bp_\Psi(\mu,0)\cong\bp_\Psi(\mu,r),$$ it suffices to prove that
\begin{equation} \label{enoughb}\bp_\Psi(\mu,0)\in\Ob\mathcal G_2[\Gamma(\mu,\Psi)],\qquad \forall\, \mu\in P^+,\end{equation} or equivalently that \begin{equation}\label{enough}
 [\bp_\Psi(\mu,0):\ev_s V(\nu)]\ne 0\, \implies\, (\nu,s)\in\Gamma(\mu,\Psi).\end{equation}
 Now,  $$[\bp_\Psi(\mu,0):\ev_s V(\nu)]\ne 0\iff\Hom_{\lie g}(\bp_\Psi(\mu,0)[s], V(\nu))\ne 0,$$ and hence \eqref{enough} follows if we prove that for any $(\nu,s)\in \Lambda$ and
 we have \begin{equation}\label{enougha}
 0\not=v \in \bp_\Psi(\mu,0)[s]_\nu,\quad \lie n^+\, v=0\,\implies \, (\nu,s)\in\Gamma(\mu,\Psi).
 \end{equation}
 By the PBW theorem, we have a decomposition of $\bz_+$-graded vector spaces
 $$\bp_\Psi(\mu,0)=\bu(\lie n^-)\bu((\lie n^-_\Psi)_{\ad})\bop_{\mu,0}= \bu((\lie n^-_\Psi)_{\ad})\bop_{\mu,0}\oplus\lie n^-\bu((\lie n^-)\bu((\lie n^-_\Psi)_{\ad})\bop_{\mu,0} .$$
 If  $v$ satisfies the assumptions in the left hand side of \eqref{enougha} then Lemma~\ref{PT1.70} implies that
  $v$ has a non--zero projection onto the subspace $\big(\bu((\lie n^-_\Psi)_{\ad})\bop_{\mu,0}\big)[s]_\nu$ of $\bp_\Psi(\mu,0)[s]_\nu$.
  Fix a numbering $\{\beta_1,\dots,\beta_\ell\}$ of $\Psi$ and note that $\big(\bu((\lie n^-_\Psi)_{\ad})\bop_{\mu,0}\big)[s]_\nu$
  is spanned by elements of the form $$(x^-_{\beta_1})_{\ad}^{r_1}\cdots (x^-_{\beta_\ell})_{\ad}^{r_\ell}\bop_{\mu,0},\qquad (r_1,\dots,r_\ell)\in\bz_+^\ell,\quad \sum_{k=1}^\ell r_k=s,\quad
  \nu =\mu-\sum_{k=1}^\ell r_k\beta_k,$$ which proves that $(\nu,s)\in\Gamma(\mu,\Psi)$.
  \end{pf}

\subsection{}\label{PT2.5}
For $(\mu,r), (\nu,s)\in\Gamma$ set \begin{gather*}c_{(\nu,s)}^{(\mu,r)}=\dim\Hom_{\lie g}(V(\nu),\bigwedge^{s-r} \lie g\tensor V(\mu))\\
d_{(\nu,s)}^{(\mu,r)}=\dim\Hom_{\lie g}(V(\nu), S^{s-r}(\lie g)\tensor V(\mu)).\end{gather*}
where we use the convention $c^{(\mu,r)}_{(\nu,s)}=d_{(\nu,s)}^{(\mu,r)}=0$ if~$s<r$.
\begin{lem} For $(\mu,r), (\nu,s)\in\Gamma$ with $(\nu,s-r)\in\Gamma(\mu,\Psi)$, we have \begin{gather}\label{cmunu} c_{(\nu,s)}^{(\mu,r)}=\dim\{ v\in (\bigwedge^{s-r} \lie n^-_\Psi)_{\nu-\mu}\,:\, (x_i^+)^{\nu(h_i)+1}(v)=0,\,\forall\, i\in I\},\\
\label{dmunu}d_{(\nu,s)}^{(\mu,r)}= \dim\{v \in S^{s-r}(\lie n^-_\Psi)_{\nu-\mu}\,:\, (x_i^-)^{\nu(h_i)+1}(v)=0,\,\forall\, i\in I\,\}.\end{gather}
 In particular, $c_{(\nu,s)}^{(\mu,0)}=c_{\nu,s}^\mu$.
 \end{lem}
\begin{pf}
Using a standard vector space isomorphism and \cite{PRV} we have \begin{align*} \Hom_{\lie g}(V(\nu),\bigwedge^{s-r} \lie g\tensor V(\mu))&\cong\Hom_{\lie g}(V(\nu)\tensor V(\mu)^*,\bigwedge^{s-r} \lie g),\\ &\cong\{ v\in (\bigwedge^{s-r} \lie g)_{\nu-\mu}\,:\, (x_i^-)^{\nu(h_i)+1}(v)=0,\,\forall\,i\in I\}.
\end{align*}
Hence \eqref{cmunu} follows if we prove that \begin{equation}\label{eq}(\bigwedge^{s-r} \lie g)_{\nu-\mu}=(\bigwedge^{s-r} \lie n^-_\Psi)_{\nu-\mu}.\end{equation}
Observe that
$(\bigwedge^{s-r} \lie g)_{\nu-\mu}$ is spanned by monomials $x_1\wedge\cdots\wedge x_{s-r}$, where $x_i\in \lie g_{\gamma_i}$,
with $\gamma_i\in R\cup\{0\}$ and $$\sum_{i=1}^{s-r} \gamma_i=\nu-\mu=-\beta_{1}-\cdots-\beta_{s-r},\qquad \beta_{k}\in\Psi,\, 1\le k\le s-r .$$  Using~\eqref{rigid} we see  that $\gamma_k\in -\Psi$ for all $1\le k\le s$ and \eqref{eq} (and hence \eqref{cmunu}) is proved. The proof of \eqref{dmunu} is similar and we omit the details.
\end{pf}

\subsection{}\label{PT2.10} We shall need the following result which was established in \cite[Propositions  3.3]{CG1}.
\begin{prop}\label{ext} Let $(\mu,r),(\nu,s)\in\Gamma$.
\begin{enumerate}[{\rm(i)}]
\item\label{PT2.10.i} We have, $$
\Ext^j_{\mathcal G_2[\Gamma]}(\ev_r V(\mu), \ev_sV(\nu))\not=0\,\implies\, j=s-r, \, (\nu,s-r)\in\Gamma(\mu,\Psi).$$
\item\label{PT2.10.ii} Suppose that $(\nu,s-r)\in\Gamma(\mu,\Psi)$. There exists an isomorphism of vector spaces
\begin{equation*}\Ext^{s-r}_{\mathcal G_2[\Gamma]}(\ev_r V(\mu), \ev_sV(\nu))\cong\Hom_{\lie g}(V(\nu),\bigwedge^{s-r} \lie g\tensor V(\mu)).\tag*{\qedsymbol}
 \end{equation*}
\end{enumerate}
\end{prop}
\subsection{}\label{PT2.20} Fix an enumeration of the set $\Gamma$.
Define $|\Gamma|\times|\Gamma|$-matrices $A(t)$ and $E(t)$ with entries in $\bz[t]$, by
\begin{align*}
&A(t)
=\Big(t^{s-r} [P(\mu,r)^\Gamma: \ev_s V(\nu)]\Big)_{(\nu,s),(\mu,r)\in\Gamma},\\
& E(t)
=\Big(t^{s-r} \dim\Ext^{s-r}_{\mathcal G_2[\Gamma]}(\ev_r V(\mu), \ev_sV(\nu))\Big)_{(\nu,s),(\mu,r)\in\Gamma}=
\Big(t^{s-r} c^{(\mu,r)}_{(\nu,s)}\Big)_{(\nu,s),(\mu,r)\in\Gamma},
\end{align*}
where we have used Proposition~\ref{PT2.10}\eqref{PT2.10.ii} and Lemma~\ref{PT2.5}.

\begin{pf}[Proof of Theorem~\ref{thm2}]
It was proved in \cite[Proposition 2.6 and 3.8]{CG1} that $$A(t)E(-t)={\rm Id}.$$
Therefore, for all $(\nu,s),(\kappa,p)\in\Gamma$
$$
\sum_{(\mu,r)\in\Gamma} t^{s-r} [P(\mu,r)^\Gamma:\ev_s V(\nu)] (-t)^{r-p} c^{(\kappa,p)}_{(\mu,r)}=\delta_{(\nu,s),(\kappa,p)}.
$$
In particular, taking $(\kappa,p)=(\lambda,0)$ and using Lemma~\ref{PT2.5} we obtain
$$
t^s \sum_{(\mu,r)\in\Gamma} (-1)^r [P(\mu,r)^\Gamma:\ev_s V(\nu)] c^{\lambda}_{\mu,r}=\delta_{(\nu,s),(\lambda,0)}.
$$
Multiplying both sides by $\ch V(\nu)$ and taking the sum over all $(\nu,s)\in\Gamma$ we conclude that
\begin{equation}\label{eq:PT2.20}
\sum_{(\mu,r)\in\Gamma} (-1)^r c^{\lambda}_{\mu,r} \ch_t P(\mu,r)^{\Gamma}=\ch V(\lambda).
\end{equation}
Applying Corollary~\ref{corthm1} completes the proof of Theorem \ref{thm2}.
\end{pf}

\section{Conjecture~\ref{subs:NN conj} and characters of  minimal affinizations}\label{maff}

Our interest in the modules $P(\mu,r)^{\Gamma}$ stems from the study of finite-dimensional representations of quantum loop algebras.
In the special case when $\lie g$ is a classical Lie algebra and $\mu=m\omega_i$, $i\in I$, we shall see that the modules $P(\mu,r)^{\Gamma}$ are the specialization of
the famous Kirillov-Reshetikhin modules. Conjecture~\ref{subs:NN conj} is concerned with relating the case of an arbitrary $\mu$ to generalizations of the Kirillov-Reshetikhin modules, called minimal affinizations. These were introduced in \cite{Ch1} and studied further in \cite{CP1,CP2,CP3}. More  recent results which are pertinent to this section can be found in \cite{Her2,M,NN1}.
  We begin this section by briefly explaining these results. The interested reader is referred to \cite{CH} and the references in that paper for further details on the representation theory of quantum loop algebras.
\subsection{} Let $\bc[t,t^{-1}]$ be the ring of Laurent polynomials in a variable $t$ and let $L(\lie g)$ be the loop algebra of $\lie g$, i.e $L(\lie g)=\lie g\otimes\bc[t,t^{-1}]$
with the Lie bracket given by $$
[x\otimes f, y\otimes g]=[x,y]\otimes fg,\qquad x,y\in\lie g, \, f,g\in\bc[t, t^{-1}].
$$
We identify $\lie g$ with the subalgebra $\lie g\otimes 1$ of $L(\lie g)$.
Observe that we have a natural isomorphism of Lie algebras
$$\lie g\ltimes (\lie g)_{\ad}\cong \lie g\otimes\bc[t]/(t-a)^2\cong \lie g\otimes\bc[t,t^{-1}]/(t-b)^2,\quad a\in\bc, \, b\in\bc^\times.$$
If we regard $\bc[t]$ as being graded by powers of $(t-a)$, then the first isomorphism above is one of $\bz_+$-graded Lie algebras.

\subsection{} Let $q$ be an indeterminate and let $\bu_q(\lie g)$ and $\bu_q(L(\lie g))$ be the corresponding quantized enveloping algebras defined over the field $\bc(q)$ of rational functions in $q$.
These algebras admit an $\ba$-form (where $\ba=\bc[q,q^{-1}]$), namely free $\ba$-submodules $\bu_\ba(\lie g)$, $\bu_\ba(L(\lie g))$ such that $$\bu_q(\lie g)\cong\bc(q)\otimes_\ba\bu_\ba(\lie g),\qquad \bu_q(L(\lie g))\cong\bc(q)\otimes_\ba\bu_\ba(L(\lie g)).$$ As a result, one can specialize $q$ to be a non-zero complex number $\epsilon$ as follows. Given $\epsilon\in \bc^\times$, let
$\bc_\epsilon=\ba/(q-\epsilon)$
and set $$\bu_\epsilon(\lie g)=\bc_\epsilon\otimes_\ba\bu_\ba(\lie g),\qquad \bu_\epsilon(L(\lie g))=\bc_\epsilon\otimes_\ba\bu_\ba(L(\lie g)).$$ If $\epsilon=1$, then $\bu_1(\lie g)$ and $\bu_1(L(\lie g))$ are (essentially) the universal enveloping algebras of $\lie g$ and $L(\lie g)$ respectively.

\subsection{} The theory of  integrable representations  of $\bu_q(\lie g)$ (and also of $\bu_\epsilon(\lie g)$ if $\epsilon$ is not a primitive root of unity) is identical to that of $\bu(\lie g)$ (see~\cite{Lus}). Namely any integrable representation is completely reducible, the isomorphism classes of irreducible representations are parametrized by elements of $P^+$ and the (suitably defined)  character of an irreducible representation  $V_q(\lambda)$ is given by the Weyl character formula. Moreover $V_q(\lambda)$ admits an $\ba$-form and one can define in the obvious way a representation $V_\epsilon(\lie g)$ of the algebra $\bu_\epsilon(\lie g)$, where $\epsilon\in\bc^\times$. If $\epsilon$ is not a root of unity, then $V_\epsilon(\lie g)$ is an irreducible representation of $\bu_\epsilon(\lie g)$ and $V_1(\lambda)$ is the usual finite-dimensional representation of $\bu(\lie g)$ with highest weight $\lambda$, i.e. $V_1(\lambda)\cong V(\lambda)$.

\subsection{}
On the other hand,
the theory of integrable representations of $\bu_q(L(\lie g))$ is  quite different from that of $\bu(L(\lie g))$. The one point of similarity is the classification of irreducible finite-dimensional representations of $\bu_q(L(\lie g))$ and $\bu(L(\lie g))$ (see \cite{C,CPnew,CPqa}). Thus the isomorphism classes of irreducible finite-dimensional representations of  $\bu_q(\lie g)$ (respectively  $\bu(L(\lie g))$) are
parametrized by $I$-tuples of polynomials with constant term~$1$, called the Drinfeld polynomials, in $\bc(q)[u]$ (respectively  $\bc[u]$) where $u$ is an indeterminate. It was shown in \cite{CPweyl} that if the Drinfeld polynomials
$\bpi=(\pi_i)_{i\in I}$ satisfy $\pi_i\in\ba[u]$, then the corresponding irreducible representation admits an $\ba$-form. Hence we can specialize $q$ to one and we get a representation of $\bu(L(\lie g))$ but in this case, the the specialized representation is usually reducible.

Much of the literature on the subject revolves around determining characters of irreducible representations of quantum loop algebras. For the purposes of this paper, we shall be interested in the $\bu_q(\lie g)$-character of an irreducible  finite-dimensional representation $V$ of $\bu_q(L(\lie g))$, namely in the multiplicities $[V:V_q(\lambda)]=\dim\Hom_{\bu_q(\lie g)}(V(\lambda), V)$, for $\lambda\in P^+$.
 This has proved to be a very hard problem and the strategy has been to identify interesting families of irreducible representations and focus on determining their character.

\subsection{} One such family are the so--called Kirillov-Reshetikhin modules, which appeared originally in
\cite{KR} and were motivated by their work on solvable lattice models.
 It is now customary to call a member of the two parameter family of
isomorphism classes of irreducible modules for $\bu_q(L(\lie g))$ with Drinfeld polynomials
  $$\bpi_{i,m}=(\pi_j)_{j\in I},\qquad \pi_j=\begin{cases} %
\prod\limits_{k=0}^{m-1} (1-q^{\frac12(\alpha_i,\alpha_i)(m-1-2k)}u)
,& j=i,\\
1,& \text{otherwise,}\end{cases}$$
a Kirillov-Reshetikhin module and it is denoted by $V(\bpi_{i,m})$.

 The characters of these modules were conjectured in~\cite{KR} and the  conjecture was established in \cite{Ch2} except for a few nodes of the Dynkin diagram of $E_n$, $n=6,7,8$.
In the course of proving the conjecture, it was shown (\cite[Theorem~2 and Corollary~2.1]{Ch2}) that  for $i\in I$ satisfying $\epsilon_i(\theta)\le 2$,
the $q=1$ specialization $V_1(\bpi_{i,m})$
of $V(\bpi_{i,m})$ is actually a module for $\lie g\otimes\bc[t]/((t-1)^2)$. Note that the condition $\epsilon_i(\theta)\le 2$ holds for all $i\in I$ if $\lie g$ is a classical Lie algebra.

 Furthermore if  $\varphi$
is the automorphism of $\lie g\tensor \bc[t]$ defined by $x\otimes t\to x\otimes(t+1)$, then $\varphi^*V_1(\bpi_{i,m})$ is a graded module for $\lie g\otimes\bc[t]$ and hence for the graded quotient $$\lie g\otimes\bc[t]/(t^2)\cong \lie g\ltimes\lie g_{\ad}.$$
More precisely, it was shown that if $\epsilon_i(\theta)=1$ then $\varphi^* V_1(\bpi_{i,m})\cong_{\mathcal G_2} \ev_0 V(m\omega_i)$ while for $i\in I$ with $\epsilon_i(\theta)=2$,
$\varphi^*V_1(\bpi_{i,m})$ is generated,  as a $\lie g\ltimes \lie g_{\ad}$-module, by an element $v_{i,m}$ subject to the relations
\begin{gather}\label{rel-KR.1}\lie n^+ v_{i,m}=0,\,\,
hv_{i,m}=m\omega_i(h) v_{i,m},\quad h\in\lie h,\,\,
(x^-_{\alpha_i})^{m+1} v_{i,m}=0=x^-_{\alpha_j} v_{i,m},\quad j\not=i\in I\\
\lie n^+_{\ad} v_{i,m}=0=\lie h_{\ad} v_{i,m},\quad
(x^-_{\alpha_i})_{\ad}\  v_{i,m}=0.\label{rel-KR.2}\end{gather}

\subsection{}\label{MA.30} As an application of Theorem \ref{thm1} we obtain a homological interpretation of the module $\varphi^*V_1(\bpi_{i,m})$.
Let $\lie g$ be a classical Lie algebra and given $i\in I$ set
$$
\Psi_i=\{\alpha\in R^+: \epsilon_i(\alpha)=2\}.
$$
Then either $\Psi_i=\emptyset$ or $\Psi_i=\{ \beta\in R\,:\, (\omega_i,\beta)=\max_{\alpha\in R} (\omega_i,\alpha)\}$.
\begin{prop}\label{krp} For all $i\in I$, $m\in \bz_+$ we have $\varphi^*V_1(\bpi_{i,m})\in\Ob\mathcal G_2$ and $$\varphi^*V_1(\bpi_{i,m})\cong_{\mathcal G_2} P(m\omega_i,0)^{\Gamma},\qquad
\Gamma=\Gamma(m\omega_i,\Psi_i).$$
 In particular, $\varphi^*V_1(\bpi_{i,m})$ is the projective cover of~$\ev_0 V(m\omega_i)$ in the category $\mathcal G_2[\Gamma]$.
\end{prop}
\begin{proof}
Clearly if $\Psi_i=\emptyset$ then $\Gamma=\{(m\omega_i,0)\}$ and the assertion is trivial. Suppose that $\Psi_i\not=\emptyset$. Then
by Theorem~\ref{thm1} it is sufficient to prove that $$\varphi^* V_1(\bpi_{i,m})\cong_{\mathcal G_2}
\bp_{\Psi_i}(m\omega_i,0),$$ where $\bp_{\Psi_i}(m\omega_i,0)$ was defined in Proposition~\ref{defining}. Since $\alpha_i\notin\Psi_i$, the
assignment $v_{i,m}\mapsto \bop_{m\omega_i,0}$ defines a surjective morphism
$\varphi^* V_1(\bpi_{i,m})\to\bp_{\Psi_i}(m\omega_i,0)$ of objects of~$\mathcal G_2$.  The proposition follows if we prove that $v_{i,m}$ satisfies the defining relations of $\bop_{m\omega_i,0}$ and for this, we only
need to show that
\begin{equation}\label{tmp.1}
(x^-_\alpha)_{\ad}v_{i,m}=0,\qquad \alpha\in R^+\setminus\Psi_i.
\end{equation}
The argument is by induction on~$\Ht(\alpha)$.
Note first that $\alpha_j\notin\Psi_i$ for all~$j\in I$. If~$j\not=i$ we have $2(x^-_{j})_{\ad}\,v_{i,m}=[x_{j}^-,(h_j)_{\ad}]\,v_{i,m}=0$ by
\eqref{rel-KR.1}, \eqref{rel-KR.2}, while for $j=i$,  \eqref{tmp.1} is just the last relation in~\eqref{rel-KR.2}.
Thus, the induction begins. For the inductive step,
choose $j\in I$ such that $\beta=\alpha-\alpha_j\in R^+$. Then clearly $\beta\notin\Psi_i$ as well and since $x_\alpha^- = c[x_{j}^-,x_\beta^-]$ for some $c\in\bc^\times$  we can write,
using the defining relations
$$(x^-_\alpha)_{\ad}\, v_{i,m}= \begin{cases}c\, x_{j}^-\, (x^-_\beta)_{\ad}\, v_{i,m},& j\ne i,\\ c\, (x^-_{i})_{\ad}\, x_{\beta}^-\, v_{i,m},& j=i.\end{cases}
$$
In the first case, the right hand side is zero by the induction hypothesis. In the second case, it is also zero since then $\epsilon_i(\beta)=0$
and so $x_\beta^-$ can be written as a commutator of the $x_{j}^-$ with $j\not=i$. Then $x^-_\beta v_{i,m}=0$ by \eqref{rel-KR.1}. This completes the inductive step.
\end{proof}
\begin{cor}
Let~$i\in I$, $m\in\bz_+$. Then
$$
\ch_t \varphi^* V_1(\bpi_{i,m})=\sum_{(\mu,s)\in\Gamma(m\omega_i,\Psi_i)} t^s d^{m\omega_i}_{\mu} \ch V(\mu),
$$
where
$$
d^{\lambda}_{\mu}=\dim\{ v\in S(\lie n^-_{\Psi_i})_{\mu-\lambda}\,:\, (x_i^-)^{\mu(h_i)+1}(v)=0,\,\forall i\in I\}.
$$
\end{cor}

\subsection{} From the mathematical point of view, it is obvious to ask if there are analogs of the Kirillov-Reshetikhin modules for a general $\lambda\in P^+$. The notion of a minimal affinization was introduced in \cite{Ch1} to answer exactly this question. We say that a simple $\bu_q(L(\lie g))$-module $V(\bpi)$ associated to Drinfeld polynomials $\bpi\in (\bc(q)[u])^{|I|}$ is an  affinization of $V_q(\lambda)$,
$\lambda\in P^+$ if
$$[V(\bpi):V_q(\lambda)]=1,\qquad [V(\bpi):V_q(\mu)]\ne 0\, \implies\,\mu\le\lambda.$$
 In fact, this happens if and only if $\lambda=\sum_{i\in I} (\deg\pi_i)\omega_i$.
We say that $V(\bpi)$ is a minimal affinization of $V_q(\lambda)$ if given any other affinization $V(\bpi')$ of $V_q(\lambda)$, one of the following holds:
\begin{enumerate}[{\rm(i)}]
\item $V(\bpi')$ is isomorphic to $V(\bpi)$ as a $\bu_q(\lie g)$-module or
\item for all $\nu\in P^+$ either $[V(\bpi):V_q(\nu)]\le [V(\bpi'):V_q(\nu)] $ or  there exists $\nu'> \nu$ such that $[V(\bpi):V_q(\nu')]< [V(\bpi'):V_q(\nu')]$.
\end{enumerate}
It was proved in \cite{Ch1} that there are finitely many $\bu_q(\lie g)$-isomorphism classes of minimal affinizations associated with a given $\lambda$. In the case when $\lambda=m\omega_i$, it was proved that the Kirillov-Reshetikhin module $V(\bpi_{i,m})$ is indeed minimal. The more general problem of determining $\bpi$ so that $V(\bpi)$ was a minimal affinization was addressed in \cite{CP1,CP2}. The results were complete  except in the case when $\lie g$ is of type $D_n$ or $E_n$,
where there are significant problems arising from the trivalent node.

\subsection{} There is recent literature on the subject of minimal affinizations.  The results of \cite{CPweyl} show that one can specialize a  minimal affinization $V(\bpi)$  of $\lambda$ to get an indecomposable representation $V_1(\bpi)$ of $L(\lie g)$  and a graded representation $\varphi^*V_1(\bpi)$ of $\lie g\tensor \bc[t]$. In \cite{M} a partial result, similar to the one for Kirillov-Reshetikhin modules discussed above was given. We formulate it in the language of this paper, the translation from the language of \cite{M} is the same as the one given for $m\omega_i$ in Proposition \ref{krp}.
Suppose that $\lie g$ is of type $B_n$, $C_n$ or~$D_{n+1}$, with the numbering of the nodes in the Dynkin diagram as in~\cite{Bo},
and suppose for simplicity that $\lambda\in P^+$ is such that $\lambda(h_i)=0$, $i\ge n$. Set
$$
i_\lambda=
\max\{i\in I\,:\, \lambda(h_i)>0\}
$$
and let $\Psi_\lambda:=\Psi_{i_\lambda}$ as defined in Section~\ref{krp}.

Then we have the following:
\begin{prop}
The $\lie g\tensor \bc[t]$-module $\varphi^*V_1(\bpi)$ satisfies $$(\lie g\otimes t^2\bc[t])\varphi^*V_1(\bpi)=0.$$ In particular, we can regard $\varphi^*V_1(\bpi)$ as a module for $\lie g\ltimes\lie g_{\ad}$ and moreover,  in this case it is a quotient of $\bp_{\Psi_\lambda}(\lambda,0)$ in the category~$\mathcal G_2$.
\end{prop}
It is conjectured in~\cite{M} (with a weaker restriction on~$\lambda$) that in fact $\varphi^* V_1(\bpi)\cong \bp_{\Psi_\lambda}(\lambda,0)$.
To prove this conjecture it is sufficient to show that the two modules have the same $\lie g$-character and this brings us naturally to another recent conjecture and  result on minimal affinizations.  In~\cite{NN1}
W.~Nakai and T.~Nakanishi conjectured that the $\bu_q(\lie g)$-character of a minimal affinization  of $V_q(\lambda)$
is given, as an element of $\bz[P]$, by the Jacobi-Trudi determinant $\bh_\lambda$ (see Section \ref{subs:NN conj}).
 In the case when $\lambda=m\omega_i$ the conjecture is known to be true through the work of~\cite{KNH,KOS,KTs}, \cite{Nak1} and \cite{Her1} (cf.~\cite{NN1} for the details). Together with Proposition \ref{krp} we see that Conjecture \ref{nnmcg} is true for $\lambda=m\omega_i$.
The conjecture is  also proved for~$\lie g$ of type~$B_n$ and $\lambda\in P^+$ with $\lambda(h_n)=0$  in \cite{Her2}. Using this result Moura was able to prove his conjecture when $i_\lambda\le 3$.

\subsection{}  It should be clear by now that Conjecture \ref{cgconj} is an amalgamation of the conjectures of Moura, Nakai-Nakanishi and Theorem \ref{thm1} of this paper. One way to prove Conjecture~\ref{cgconj} would be to proceed by induction on $i_\lambda$. We will need the following easily verified description of the set $\Psi_i$, $1\le i<n$.
\begin{equation}\label{psii}
\Psi_i=\begin{cases}
\{ \omega_r+\omega_s-\omega_{r-1}-\omega_{s-1}\,:\, 1\le r<s\le i\},& \lie g=B_n, D_{n+1}\\
\{ \omega_r+\omega_s-\omega_{r-1}-\omega_{s-1}\,:\, 1\le r\le s\le i\},&\lie g=C_n,
\end{cases}
\end{equation} 
where we set $\omega_0=0$.
As a consequence, we have
\begin{lem}\label{iinc} Let $\lambda\in P^+$ be such that $i_\lambda<n$ and let $(\mu,s)\in\Gamma(\lambda,\Psi_\lambda)$. Then $i_\mu\le i_\lambda$.\qedhere\end{lem}

\subsection{} If $i_\lambda=1$ then $\lambda=k\omega_1$ and  the result is known. For $i_\lambda=i$, we  proceed by a further induction (with respect to $\le$) on $\lambda$. If $\lambda\in P^+$ is minimal with $i_\lambda=i$, then for all $(\mu,s)\in\Gamma(\lambda,\Psi_\lambda)$ we have by Lemma \ref{iinc} that  $i_\mu<i$ and hence the induction hypothesis applies. The idea now is to use
  the Koike-Terada formulae \cite[Theorems~1.3.2, 1.3.3]{KT}
$$
\ch V(\lambda)=\left\{\begin{array}{ll} \det\big( \sum_{r=0}^j \boh_{\lambda_i-i-j+2r}\big)_{1\le i,j\le i_\lambda},&\lie g=B_n, D_{n+1},\\
\det\big(\boh_{\lambda_i-i+j}-\boh_{\lambda_i-i+j-2}\\\hphantom{\det\big(}+(1-\delta_{j,1})(\boh_{\lambda_i-i-j+2}-\boh_{\lambda_i-i-j})\big)_{1\le i,j\le i_\lambda},& \lie g=C_n,\end{array}\right.
$$
for the character of a simple $\lie g$-module to check that \eqref{nnmcg} holds
as an identity in the ring $\bz[\boh_k\,:\, k\in\bz_+]$.

So far, we have used a computer program to check this for $\lambda\in P^+$ with $i_\lambda\le 5$. As an example of the computation, suppose that
$\lie g$ is  of type $B_n$, $D_{n+1}$ and that $i_\lambda=3$.  Then $$\Psi_\lambda=\Psi_3=\{\omega_2,\omega_1+\omega_3-\omega_2, \omega_3-\omega_1\}.
$$
Since $|\Psi_\lambda|=3$, we conclude that $\bigwedge^s \lie n_{\Psi_\lambda}^-=0$ if $s>3$. Since the set $\Psi_\lambda$ is linearly independent in this case,
$\eta\in \wt \bigwedge\lie n^-_{\Psi_\lambda}$ if and only if $\eta=\sum_{\beta\in S} \beta$ for some $S\subset\Psi_\lambda$ and
all weight spaces of $\bigwedge\lie n^-_{\Psi_\lambda}$ are one-dimensional.
The definition of $c_{\mu,s}^\lambda$ implies that the left hand side of \eqref{nnmcg} contains at most eight terms, $c_{\mu,s}^\lambda=0$ if $s>3$ and
$c_{\mu,s}^\lambda\le 1$. Finally, it is easy to check that $c_{\mu,s}^\lambda=1$ if and only if $\mu=\lambda-\sum_{\beta\in S} \beta\in P^+$ where
$S\subset \Psi_\lambda$ and $|S|=s$. We list the elements $(\mu,s)\in\Gamma(\lambda,\Psi_\lambda)$ for which $c_{\mu,s}^\lambda=1$ below
\begin{alignat*}{2}
&(\lambda,0)& \\
 &(\lambda-\omega_2,1)& \qquad &\lambda(h_2)\ge 1\\
 &(\lambda-\omega_1+\omega_2-\omega_3,1)& &\lambda(h_1),\lambda(h_3)\ge 1\\
 &(\lambda+\omega_1-\omega_3,1)& &\lambda(h_3)\ge 1\\
 &(\lambda-\omega_1-\omega_3,2)& &\lambda(h_1),\lambda(h_3)\ge 1\\
 &(\lambda+\omega_1-\omega_2-\omega_3,2)&  &\lambda(h_2),\lambda(h_3)\ge 1\\
 &(\lambda+\omega_2-2\omega_3,2)& &\lambda(h_2)\ge 2\\
 &(\lambda-2\omega_3,3)& &\lambda(h_3)\ge 2.
\end{alignat*}
Suppose that $\lambda$ satisfies $i_\lambda=4$. We have $|\Psi_4|=6$ and so the left hand side of~\eqref{nnmcg} contains 
at most $64$ terms. In fact, $|\wt \bigwedge \lie n^-_{\Psi_3}|=54$ and 6 weight spaces have dimension $2$ while 2 weight spaces are of dimension~3. 
In the table below we list the elements $(\mu,s)$ of $\Gamma(\lambda,\Psi)$ which contribute to the sum in~\eqref{nnmcg} together with the 
values of $c^\lambda_{\mu,s}$. To shorten the notation,
we use the convention that $c^\lambda_{\mu,s}=0$ if $\mu\notin P^+$. We have 
{\small
\begin{alignat*}{4}
    &(\lambda,0) &\quad& 1 &\quad
    &(\lambda+\omega _1-\omega _3,1) &\quad& 1 \\
    &(\lambda+\omega _1-\omega _2+\omega _3-\omega _4,1) && 1 &
    &(\lambda+\omega _2-\omega _4,1) && 1 \\
    &(\lambda-\omega _2,1) && 1 &
    &(\lambda-\omega _1+\omega _2-\omega _3,1) && 1 \\
    &(\lambda-\omega _1+\omega _3-\omega _4,1) && 1 &
    &(\lambda+2 \omega _1-\omega _2-\omega _4,2) && 1 \\
    &(\lambda+\omega _1+\omega _2-\omega _3-\omega _4,2) && 1 &
    &(\lambda+\omega _1+\omega _3-2 \omega _4,2) && 1 \\
    &(\lambda+\omega _1-\omega _2-\omega _3,2) && 1 &
    &(\lambda+\omega _1-2 \omega _2+\omega _3-\omega _4,2) && 1 \\
    &(\lambda-\omega _4,2) && \sum\nolimits_{i=1}^3 \delta(\lambda(h_i)\ge 1)&&(\lambda+\omega _2-2 \omega _3,2) && 1 
      \\
    &(\lambda-\omega _2+2 \omega _3-2 \omega _4,2) && 1 &
    &(\lambda-\omega _1+2 \omega _2-\omega _3-\omega _4,2) && 1 \\
    &(\lambda-\omega _1+\omega _2+\omega _3-2 \omega _4,2) && 1 &
    &(\lambda-\omega _1-\omega _3,2) && 1 \\
    &(\lambda-\omega _1-\omega _2+\omega _3-\omega _4,2) && 1 &
    &(\lambda-2 \omega _1+\omega _2-\omega _4,2) && 1 \\
    &(\lambda+2 \omega _1-2 \omega _4,3) && 1 &
    &(\lambda+2 \omega _1-2 \omega _2-\omega _4,3) && 1 \\
    &(\lambda+\omega _1-\omega _3-\omega _4,3) & &\delta(\lambda(h_2)\geq
      1)+1 &
    &(\lambda+\omega _1-\omega _2+\omega _3-2 \omega _4,3) && 2 \\
    &(\lambda+\omega _2-2 \omega _4,3) & &\delta(\lambda(h_1)\geq
      1)+\delta(\lambda(h_3)\geq 1)     &&(\lambda+2 \omega _2-2 \omega _3-\omega _4,3) && 1 \\
    &(\lambda+2 \omega _3-3 \omega _4,3) && 1 &
    &(\lambda-2 \omega _3,3) && 1 \\
    &(\lambda-\omega _2-\omega _4,3) & &\delta(\lambda(h_1)\geq
      1)+\delta(\lambda(h_3)\geq 1) &
    &(\lambda-2 \omega _2+2 \omega _3-2 \omega _4,3) && 1 \\
    &(\lambda-\omega _1+\omega _3-2 \omega _4,3) & &\delta(\lambda(h_2)\geq
      1)+1&&(\lambda-\omega _1+\omega _2-\omega _3-\omega _4,3) && 2 \\
    &(\lambda-2 \omega _1+2 \omega _2-2 \omega _4,3) && 1 &
    &(\lambda-2 \omega _1-\omega _4,3) && 1 \\
    &(\lambda+2 \omega _1-\omega _2-2 \omega _4,4) && 1 &
    &(\lambda+\omega _1+\omega _2-\omega _3-2 \omega _4,4) && 1 \\
    &(\lambda+\omega _1+\omega _3-3 \omega _4,4) && 1 &
    &(\lambda+\omega _1-\omega _2-\omega _3-\omega _4,4) && 1 \\
    &(\lambda+\omega _1-2 \omega _2+\omega _3-2 \omega _4,4) && 1 &
    &(\lambda+\omega _2-2 \omega _3-\omega _4,4) && 1 \\
    &(\lambda-2 \omega _4,4) & & \sum\nolimits_{i=1}^3 \delta(\lambda(h_i)\ge 1) &
    &(\lambda-\omega _2+2 \omega _3-3 \omega _4,4) && 1 \\
    &(\lambda-\omega _1+2 \omega _2-\omega _3-2 \omega _4,4) && 1 &
    &(\lambda-\omega _1+\omega _2+\omega _3-3 \omega _4,4) && 1 \\
    &(\lambda-\omega _1-\omega _3-\omega _4,4) && 1 &
    &(\lambda-\omega _1-\omega _2+\omega _3-2 \omega _4,4) && 1 \\
    &(\lambda-2 \omega _1+\omega _2-2 \omega _4,4) && 1 &
    &(\lambda+\omega _1-\omega _3-2 \omega _4,5) && 1 \\
    &(\lambda+\omega _1-\omega _2+\omega _3-3 \omega _4,5) && 1 &
    &(\lambda+\omega _2-3 \omega _4,5) && 1 \\
    &(\lambda-\omega _2-2 \omega _4,5) && 1 &
    &(\lambda-\omega _1+\omega _2-\omega _3-2 \omega _4,5) && 1 \\
    &(\lambda-\omega _1+\omega _3-3 \omega _4,5) && 1 &
    &(\lambda-3 \omega _4,6) && 1,
\end{alignat*}
}
\noindent
where $\delta(P)=1$ if $P$ is true and $\delta(P)=0$ otherwise. Note that for one-dimensional weight spaces the extra condition in the definition of~$c^\lambda_{\mu,s}$ is 
always vacuous.

Let $\lie g$ be of type $C_n$ and assume that $i_\lambda=3$. We have $|\Psi_3|=6$ and so the left hand side of~\eqref{nnmcg} contains 
at most $64$ terms. In fact, $|\wt \bigwedge \lie n^-_{\Psi_3}|=51$ and 13 weight spaces have dimension $2$. 
In the table below we list the elements $(\mu,s)$ of $\Gamma(\lambda,\Psi)$ which contribute to the sum in~\eqref{nnmcg} together with the 
value of $c^\lambda_{\mu,s}$, with the same conventions as above
{\small
\begin{alignat*}{4}
    &(\lambda,0) &\quad& 1 &\qquad 
    &(\lambda+2 \omega _1-2 \omega _2,1) &\quad& 1 \\
    &(\lambda+\omega _1-\omega _3,1) &&\delta(\lambda(h_2)\geq 1) &
    &(\lambda+2 \omega _2-2 \omega _3,1) && 1 \\
    &(\lambda-\omega _2,1) &&\delta(\lambda(h_1)\geq 1) &
    &(\lambda-\omega _1+\omega _2-\omega _3,1) && 1 \\
    &(\lambda-2 \omega _1,1) && 1 &
    &(\lambda+3 \omega _1-2 \omega _2-\omega _3,2) && 1 \\
    &(\lambda+2 \omega _1-2 \omega _3,2) &&\delta(\lambda(h_2)\geq 1) &
    &(\lambda+2 \omega _1-3 \omega _2,2) && 1 \\
    &(\lambda+\omega _1-\omega _2-\omega _3,2) &&\delta(\lambda(h_1)\geq
      1)+\delta(\lambda(h_2)\geq 2) &
 &(\lambda+\omega _1+2 \omega _2-3 \omega _3,2) && 1 
 \\
    &(\lambda+\omega _2-2 \omega _3,2) &&\delta(\lambda(h_1)\geq
      1)+\delta(\lambda(h_2)\geq 1) &&(\lambda-2 \omega _1+2 \omega _2-2 \omega _3,2) && 1 
    \\
     &(\lambda-\omega _1-\omega _3,2) &&\delta(\lambda(h_1)\geq
      2)+\delta(\lambda(h_2)\geq 1)&
    &(\lambda-\omega _1+3 \omega _2-3 \omega _3,2) && 1 
    \\
    &(\lambda-2 \omega _2,2) &&\delta(\lambda(h_1)\geq 1) &
    &(\lambda-2 \omega _1-\omega _2,2) && 1 \\
    &(\lambda-3 \omega _1+\omega _2-\omega _3,2) && 1 &
    &(\lambda+3 \omega _1-3 \omega _3,3) && 1 \\
    &(\lambda+2 \omega _1-\omega _2-2 \omega _3,3) &&\delta(\lambda(h_2)\geq
      2)+1 &&(\lambda+3 \omega _1-3 \omega _2-\omega _3,3) && 1 \\
    &(\lambda+\omega _1+\omega _2-3 \omega _3,3) &&\delta(\lambda(h_1)\geq
      1)+\delta(\lambda(h_2)\geq 1) & &(\lambda-3 \omega _1-\omega _3,3) && 1 
     \\
    &(\lambda-2 \omega _3,3) & & 2 \delta(\lambda(h_1)\geq 1\land \lambda(h_2)\geq 1)  & &(\lambda+3 \omega _2-4 \omega _3,3) && 1 
      \\
    &(\lambda-\omega _1+2 \omega _2-3 \omega _3,3) &&\delta(\lambda(h_1)\geq
      2)+1 &&(\lambda-3 \omega _2,3) && 1 \\
    &(\lambda-\omega _1-\omega _2-\omega _3,3) &&\delta(\lambda(h_1)\geq
      2)+\delta(\lambda(h_2)\geq 2) &
    &(\lambda-3 \omega _1+3 \omega _2-3 \omega _3,3) && 1 \\
     &(\lambda-2 \omega _1+\omega _2-2 \omega _3,3) &&\delta(\lambda(h_2)\geq
      1)+1 \\
    &(\lambda+\omega _1-2 \omega _2-\omega _3,3) &&\delta(\lambda(h_1)\geq
      1)+1 \\
    &(\lambda+2 \omega _1+\omega _2-4 \omega _3,4) && 1   & &(\lambda+3 \omega _1-\omega _2-3 \omega _3,4) && 1 \\ 
    &(\lambda+\omega _1-3 \omega _3,4) &&\delta(\lambda(h_1)\geq
      1)+\delta(\lambda(h_2)\geq 1) & &(\lambda+2 \omega _1-2 \omega _2-2 \omega _3,4) && 1 \\
    &(\lambda+2 \omega _2-4 \omega _3,4) &&\delta(\lambda(h_1)\geq 1) &
   &(\lambda+\omega _1-3 \omega _2-\omega _3,4) && 1 
    \\
    &(\lambda-\omega _2-2 \omega _3,4) &&\delta(\lambda(h_1)\geq
      1)+\delta(\lambda(h_2)\geq 2) &
  &(\lambda-\omega _1-2 \omega _2-\omega _3,4) && 1 \\
    &(\lambda-\omega _1+\omega _2-3 \omega _3,4) &&\delta(\lambda(h_1)\geq
      2)+\delta(\lambda(h_2)\geq 1) &
    &(\lambda-2 \omega _1+3 \omega _2-4 \omega _3,4) && 1 \\
    &(\lambda-2 \omega _1-2 \omega _3,4) &&\delta(\lambda(h_2)\geq 1) &
    &(\lambda-3 \omega _1+2 \omega _2-3 \omega _3,4) && 1 \\
    &(\lambda+2 \omega _1-4 \omega _3,5) && 1 &
    &(\lambda+\omega _1-\omega _2-3 \omega _3,5) && 1 \\
    &(\lambda+\omega _2-4 \omega _3,5) &&\delta(\lambda(h_1)\geq 1) &
    &(\lambda-2 \omega _2-2 \omega _3,5) && 1 \\
    &(\lambda-\omega _1-3 \omega _3,5) &&\delta(\lambda(h_2)\geq 1) &
    &(\lambda-2 \omega _1+2 \omega _2-4 \omega _3,5) && 1 \\
    &(\lambda-4 \omega _3,6) && 1
\end{alignat*}
}
\noindent
It should be noted that in this case the extra condition in the definition of~$c^\lambda_{\mu,s}$ (cf.~\ref{MR.100}) is not vacuous even
for one-dimensional weight spaces.

These computations illustrate the general phenomenon. The sum in the left hand side of~\eqref{nnmcg} contains at most $2^{|\Psi_\lambda|}$ terms although
the number of simples occurring in $P(\lambda,0)^{\Gamma(\lambda,\Psi_\lambda)}$ and their multiplicities grow much faster. Moreover, for $\lambda\gg 0$ \eqref{nnmcg} becomes
$$
\sum_{S\subset \Psi} (-1)^{|S|} \bh_{\lambda-\sum_{\beta\in S}\beta}=\ch V(\lambda).
$$



\end{document}